%% file: Circle_Pattern_Theorem.tex
\documentclass[11pt,reqno,]{amsart}
\usepackage{dsfont}
\usepackage{mathrsfs}
\usepackage{amssymb}
\usepackage{mathtools}
\usepackage{amsfonts}
\usepackage{hyperref}
\usepackage{graphicx}
\usepackage{subcaption}
\usepackage{xcolor}
\usepackage{bbm}
\usepackage{color}
\usepackage{pict2e}
\usepackage{epsf,xy,epic,amscd}
\usepackage{xspace}
\usepackage{comment}
\usepackage{tikz-cd}

\usepackage{tikz}
\usepackage{pgfplots}
\pgfplotsset{compat=1.15}
\usepackage{mathrsfs}
\usetikzlibrary{arrows}

\definecolor{ffqqqq}{rgb}{1.,0.,0.}
\definecolor{qqffqq}{rgb}{0.,1.,0.}
\definecolor{ffffff}{rgb}{1.,1.,1.}

\theoremstyle{plain}
\newtheorem{theorem}{Theorem}[section]
\newtheorem{lemma}[theorem]{Lemma}

\newtheorem{proposition}[theorem]{Proposition}

\theoremstyle{definition}
\newtheorem{definition}[theorem]{Definition}
\newtheorem{remark}[theorem]{Remark}

\numberwithin{equation}{section}

\renewcommand{\S}{\ensuremath{\Sigma}\xspace}

\newcommand{\hp}{\ensuremath{\hat{p}}\xspace}

\newcommand{\sT}{\ensuremath{\mathcal{T}}\xspace}

\newcommand{\sK}{\ensuremath{\mathcal{K}}\xspace}

\newcommand{\Hmm}[1]{\leavevmode{\marginpar{\tiny%
$\hbox to 0mm{\hspace*{-0.5mm}$\leftarrow$\hss}%
\vcenter{\vrule depth 0.1mm height 0.1mm width \the\marginparwidth}%
\hbox to
0mm{\hss$\rightarrow$\hspace*{-0.5mm}}$\\\relax\raggedright #1}}}

\AtEndDocument{\bigskip{\footnotesize
    \noindent Aijin Lin: \texttt{linaijin@nudt.edu.cn} \\
    \textsc{College of Science, \\ National University of Defense Technology, 410073 Changsha, China} \par
    \addvspace{\medskipamount}
    \noindent Qingyi Liu: \texttt{liuqingyi@nudt.edu.cn} \\
    \textsc{College of Science, \\ National University of Defense Technology, 410073 Changsha, China} \par
}}

\hfuzz=\maxdimen
\DeclareFixedFont{\Acknowledgment}{OT1}{cmr}{bx}{n}{14pt}
\begin{document}

\title{Circle Pattern Theorem for Quasi-simplicial Triangulated Surfaces}
\author{Aijin Lin$^\ast$, Qingyi Liu}
\date{\today}

\begin{abstract}
The Circle Pattern Theorem characterizes the existence and rigidity of circle patterns
with prescribed intersection angles on simplicial triangulations of closed surfaces.
In this paper we extend the theorem to \emph{quasi-simplicial} triangulations ---
triangulations that may contain loops and multiple edges, but whose lifts to the
universal cover are simplicial. Chow and Luo first considered such triangulations---under the name \emph{generalized
triangulations} (J.~Differential Geom.~\textbf{63}(1):97--129, 2003)---but with the
strong restriction that any three vertices determine at most one triangle; this
condition keeps the combinatorics within the simplicial complex framework and
consequently excludes most quasi-simplicial triangulations. We remove this restriction, work instead with the more flexible framework of Delta
complexes, and use a finite covering technique to reduce the problem to the simplicial
case. We prove that the curvature image is completely characterized by KAT inequalities
imposed on all subsets of the lifted vertex set.

\medskip
\noindent{\bf Mathematics Subject Classifications (2020):} 52C26, 52C25.

\end{abstract}
\let\thefootnote\relax \footnotetext {$^\ast$Corresponding author}
\maketitle
\tableofcontents

\section{Introduction}
\subsection{Background}
Circle patterns were introduced by Thurston~\cite{Thurston} as a powerful tool to study hyperbolic 3-manifolds, and have since evolved into a central topic in discrete differential geometry. They provide a discrete analogue of conformal maps~\cite{Thurston2,Rodin-Sullivan} and have deep connections to combinatorics~\cite{Beardon-Stephenson}, computational geometry~\cite{Stephenson}, minimal surfaces~\cite{Bobenko-Hoffman-Springborn} and variational principles~\cite{Colin, Bobenko-Springborn}.

Let $\Sigma$ be a closed orientable surface equipped with a triangulation $\sT = (V,E,F)$. A circle pattern of $\sT$-type on $\Sigma$ is a collection of oriented circles $\{C_v\}_{v \in V}$ centered at the vertices, such that whenever there is an edge between $u$ and $w$, the circles $C_u$ and $C_w$ intersect. The exterior intersection angle on each edge $e \in E$ is denoted by $\Phi(e) \in [0,\pi)$. A fundamental question is: given a function $\Phi \colon E \to [0,\pi)$, does there exist a $\sT$-type circle pattern realizing these intersection angles? And if not, can one characterize the set of all possible discrete curvatures?

When the triangulation $\sT$ is \emph{simplicial} (i.e., contains no loops and no multiple edges), a celebrated answer is provided by the Circle Pattern Theorem, which originates in the work of Koebe~\cite{Koebe} and Andreev~\cite{Andreev}, was reinterpreted by Thurston~\cite[Chap.~13]{Thurston}, and has since been extended in several directions: Chow and
Luo~\cite{Chow-Luo} introduced the combinatorial Ricci flow and characterized its
convergence via KAT inequalities; Pepa and Popelensky~\cite{PP17} and
Zhou~\cite{Zhou17} studied circle patterns with obtuse intersection angles, with Zhou's treatment
being more systematic; Ge and Jiang~\cite{GJ19} generalized to inversive
distance and studied the corresponding combinatorial Ricci flow; and Zhang and
Zheng~\cite{ZZ24} investigated the combinatorial Ricci flow for obtuse angles. To state it, we need the following condition on the intersection angles, which originates in the work of~\cite{PP17} and~\cite{Zhou17}: for every triangle with edges $e_1, e_2, e_3$,
\begin{equation}\label{cond:S}
\cos\Phi(e_1) + \cos\Phi(e_2)\cos\Phi(e_3) \ge 0,
\end{equation}
together with its cyclic permutations. This condition not only ensures that every triple of positive radii yields a non-degenerate geometric triangle, but, more importantly, it guarantees the monotonicity properties of the inner angles with respect to the radii (Lemma~\ref{L-2-4}), which are essential for the injectivity of the curvature map.

For a non-empty proper subset $I\subset V$, let $\Sigma(I)$ be the set of all cells of $\mathcal{T}$ whose vertices lie entirely in $I$.  Its Euler characteristic is $\chi(\Sigma(I)) = |I| - |E(I)| + |F(I)|$, where $E(I)\subset E$ and $F(I)\subset F$ denote the sets of edges and faces, respectively, with all vertices in $I$.  The \emph{link} of $I$, denoted by $\operatorname{Lk}(I)$, is the set of pairs $(e,\triangle)$ with $e\in E$ and $\triangle\in F$ such that
\begin{itemize}
    \item $e$ is a side of $\triangle$ (i.e., $e\prec\triangle$);
    \item both endpoints of $e$ lie outside $I$;
    \item $\triangle$ has at least one vertex in $I$.
\end{itemize}

\begin{theorem}[Circle Pattern Theorem for simplicial triangulations]\label{T-1-1}
Let $\sT$ be a simplicial triangulation of a closed orientable surface $\Sigma$ and let $\Phi\colon E\to[0,\pi)$ satisfy~\eqref{cond:S}. Then the curvature map $\sK \colon \mathbb{R}_{>0}^V \to \mathbb{R}^V$ has the following image: a vector $(K_v)_{v\in V}$ belongs to $\sK(\mathbb{R}_{>0}^V)$ if and only if
\begin{equation}\label{eq:KAT-simplicial}
\sum_{v\in I}K_v > -\sum_{(e,\triangle)\in\operatorname{Lk}(I)}(\pi-\Phi(e)) + 2\pi\chi(\Sigma(I))
\end{equation}
for every non-empty proper subset $I\subset V$, and
\[
\sum_{v\in V}K_v = 2\pi\chi(\Sigma)\text{ in the Euclidean background},\quad
\sum_{v\in V}K_v > 2\pi\chi(\Sigma)\text{ in the hyperbolic background}.
\]
\end{theorem}

We shall refer to inequalities of the form~\eqref{eq:KAT-simplicial} as the \emph{KAT inequalities}, named after Koebe, Andreev, and Thurston. These conditions characterize exactly the feasible discrete curvatures for a simplicial triangulation.

However, the requirement that $\sT$ be simplicial excludes many natural and economical triangulations. Chow and Luo first introduced triangulations that permit loops and multiple edges into discrete differential geometry in their foundational work on combinatorial Ricci flow~\cite{Chow-Luo}; they called such triangulations---the simplest kind, namely
those that become simplicial when lifted to the universal cover---\emph{generalized triangulations} and noted that their results extend to this broader setting, but under a significant restriction that any three vertices determine at most one triangle. While this condition keeps the
combinatorics within the familiar simplicial complex framework, it also rules out most triangulations of this kind. More recently, Bonsante and Wolf~\cite{BW25} used the term \emph{quasi-simplicial triangulations} for these objects without this restriction and proved the key lemma (Lemma~\ref{lem:quasi-simplicial}) that they can always be unwrapped to simplicial triangulations on suitable finite covers; we adopt their terminology because it captures the geometric nature of these triangulations more precisely.
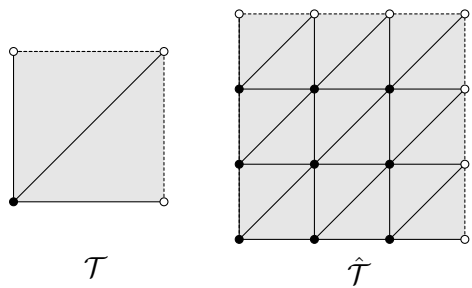
\begin{figure}[htbp]
  \centering
  \input{1vertex-T2}  
  
  \caption{Left: The one-vertex triangulation $\mathcal{T}$ of the torus. Right: Its lifted simplicial triangulation $\widehat{\mathcal{T}}$ on the $3\times 3$ covering. (Open vertices and dashed edges are copies of solid vertices and edges via the deck transformations.)}
  \label{fig:1vertex-T2}
\end{figure}

A striking illustration is the one-vertex triangulation of the torus, which uses only one vertex, three loop edges, and two triangles---a quasi-simplicial triangulation far more compact than any simplicial one (Note that this triangulation is
excluded by Chow--Luo's restriction). More generally, any closed orientable surface with Euler characteristic $\chi(\Sigma) \le 0$ admits a one-vertex triangulation. In contrast, a classical result of Jungerman and Ringel~\cite{JR80} shows that the minimal number of vertices in a simplicial triangulation of a closed orientable surface of genus $g$ is $\lceil (7+\sqrt{1+48g})/2 \rceil$ for all $g\neq 2$, while for $g=2$ the minimum is $10$; in all cases the growth is of order $\sqrt{g}$. Thus, relaxing the simplicial condition reduces the number of vertices from $O(\sqrt{g})$ down to a constant (one vertex), a drastic simplification.

To see how such a quasi-simplicial triangulation can be unwrapped, consider again the one-vertex torus. The torus is the quotient $T^2 = \mathbb{R}^2 / \mathbb{Z}^2$. Its one-vertex triangulation $\mathcal{T}$ consists of a single vertex, three loop edges, and two triangles (see Figure~\ref{fig:1vertex-T2}, left). Taking the finite covering $\widehat{T}^2 = \mathbb{R}^2 / (3\mathbb{Z} \times 3\mathbb{Z})\rightarrow T^2$, whose deck transformation group is $(\mathbb{Z}/3\mathbb{Z})^2$, yields a $3\times 3$ covering. The lifted triangulation $\widehat{\mathcal{T}}$ on $\widetilde{T}^2$ becomes simplicial: every vertex link is a hexagon, and there are no loops or multiple edges (Figure~\ref{fig:1vertex-T2}, right). This illustrates the essential phenomenon: local combinatorial singularities of a quasi-simplicial triangulation are ``unwrapped'' by a finite covering.

The existence of such a finite covering is a consequence of the residual finiteness of surface groups.  Indeed, $\pi_1(\Sigma)$ is finitely generated and residually finite, so for each combinatorial singularity (e.g., a loop or a pair of parallel edges) one can find a finite index normal subgroup whose corresponding covering unwraps that singularity.  Intersecting finitely many such subgroups yields a finite regular covering $\hat p\colon\widehat{\Sigma}\to\Sigma$ for which $\widehat{\mathcal{T}}=\hat p^*\mathcal{T}$ is simplicial.

A further motivation for quasi-simplicial triangulations comes from the geometry of hyperbolic $2$-orbifolds (see, e.g., Ratcliffe~\cite[Chapter~13]{R2019}). Let $g>0$ and let $P\subset\mathbb{H}^2$ be a convex $4g$-gon with angles $\pi/n_1,\dots,\pi/n_{4g}$ ($n_i\in\mathbb{Z}_{\ge 2}$) satisfying $\sum_{i=1}^{4g} 1/n_i < 4g-2$. Then $P = \mathbb{H}^2/\Gamma$ where $\Gamma$ is generated by the reflections in the sides of $P$; as a quotient of $\mathbb{H}^2$, $P$ carries a hyperbolic metric with boundary and corner singularities. Gluing opposite edges of $P$ yields a closed hyperbolic orbifold $\mathcal{O}$ of genus $g$. Thus $P$ has orbifold type $(;n_1,\dots,n_{4g})$, while the closed orbifold $\mathcal{O}$ obtained by gluing opposite edges has type $(2/\!\sum 1/n_i)$. The $4g$ corner vertices of $P$ merge into a single cone point $v$ with total angle $\Theta=\sum_{i=1}^{4g} \pi/n_i$, and the curvature $K=2\pi-\Theta$ satisfies $K=2\pi\chi(\mathcal{O})+\operatorname{Area}(\mathcal{O})$ by Gauss--Bonnet. Figure~\ref{fig:hypTorus} illustrates the special case $g=1$. Using only $v$ as a vertex gives a one-vertex quasi-simplicial triangulation of $\mathcal{O}$; prescribing $K$ directly encodes the cone angle, whereas adding extra vertices complicates the circle pattern problem without changing the global geometry.

\begin{figure}[htbp]
    \centering
    \begin{subfigure}{0.45\textwidth}
        \centering
        \includegraphics[width=\linewidth]{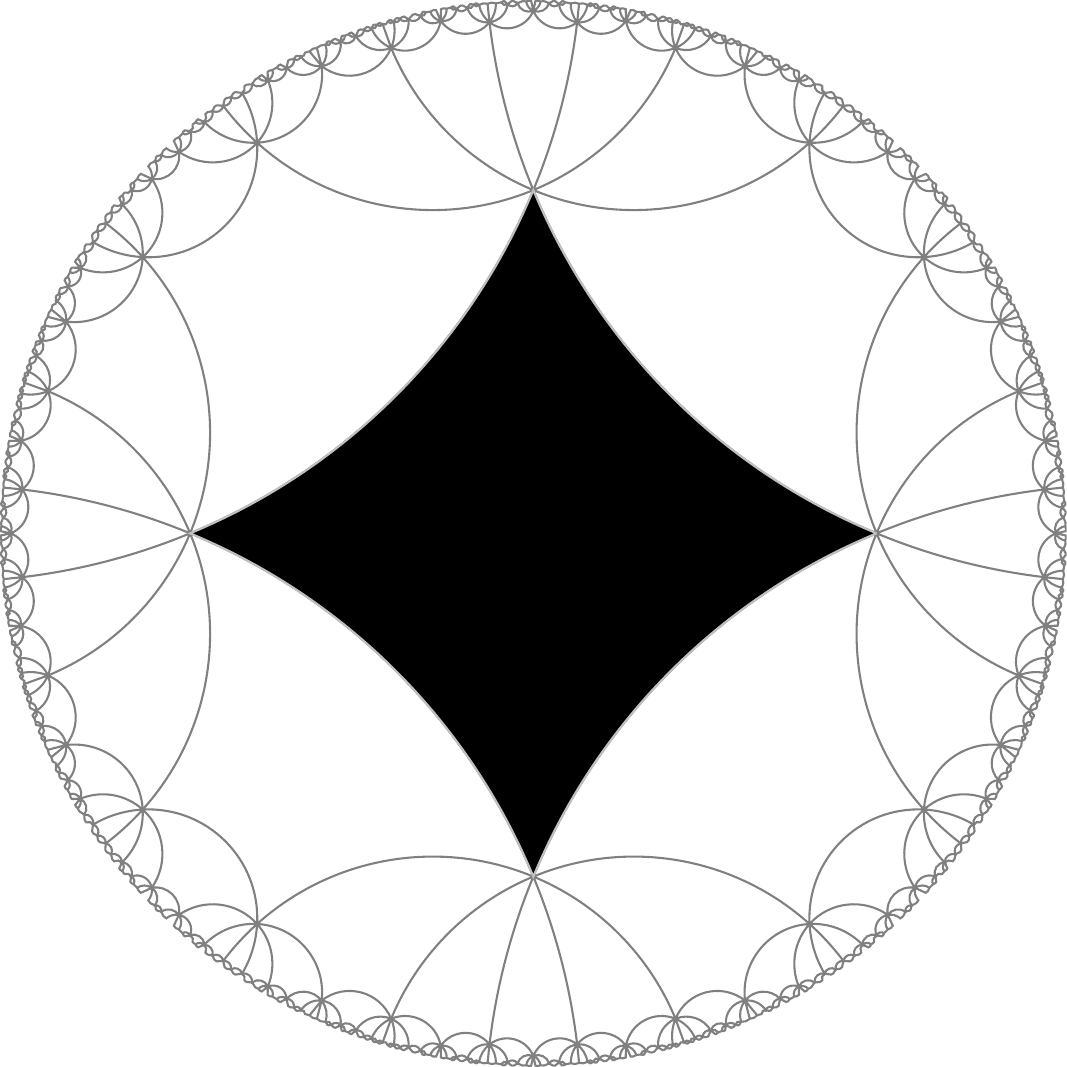} 
        \caption{The fundamental polygon $P$ in the Poincar\'e disk model of $\mathbb{H}^2$, tessellated by the reflections in its sides. Here $n_1=\dots=n_4=4$, so each corner angle is $\pi/4$.}
    \end{subfigure}
    \hfill
    \begin{subfigure}{0.45\textwidth}
        \centering
        \includegraphics[width=\linewidth]{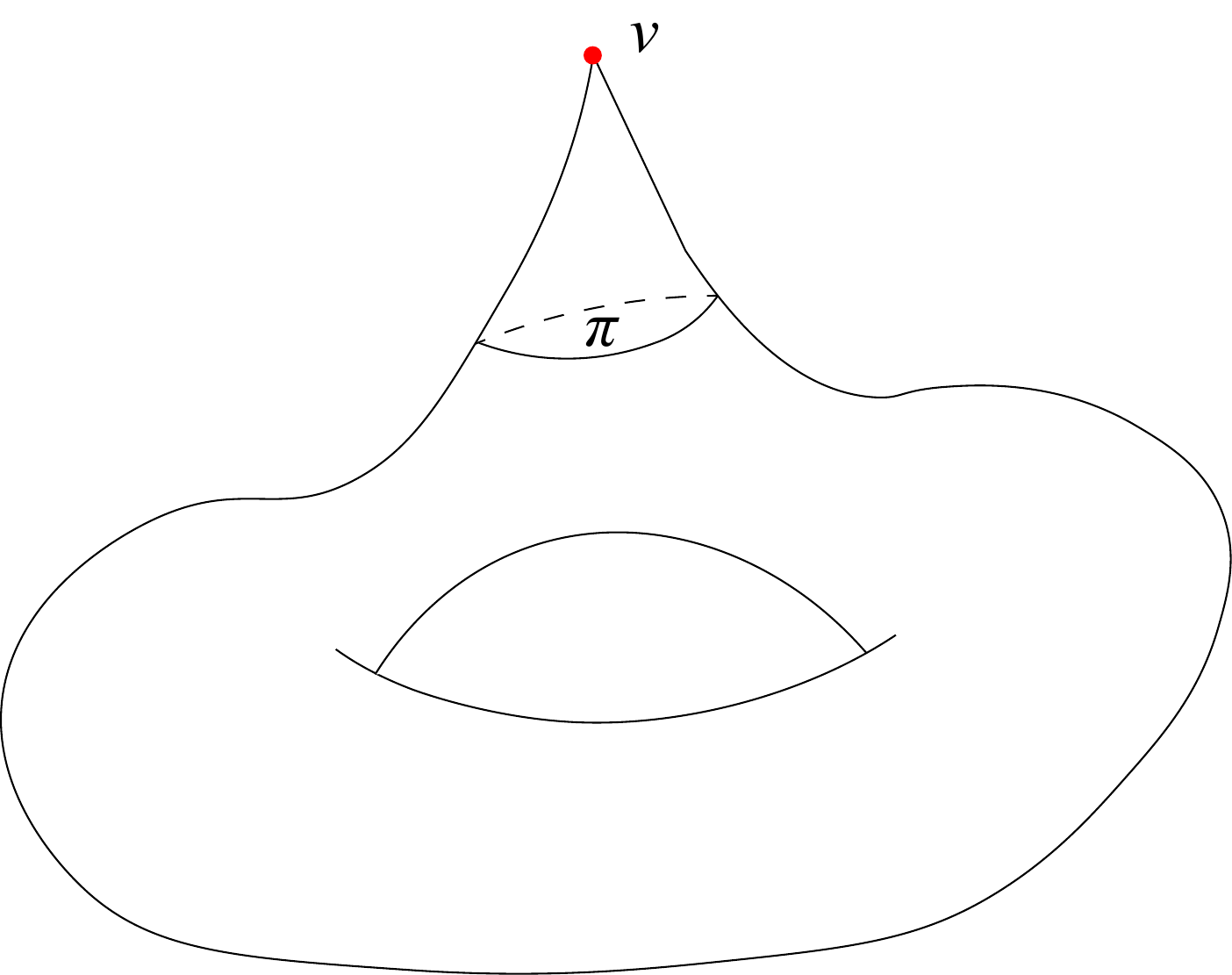} 
        \caption{A schematic view of the resulting hyperbolic torus with one cone point of angle $\pi$.}
    \end{subfigure}
    \caption{Construction of a hyperbolic orbifold of genus $g=1$ with a single cone point.}
    \label{fig:hypTorus}
\end{figure}

The above discussion suggests a natural approach to constructing orbifold metrics via circle patterns. Given a closed orientable surface $\Sigma$ with marked points $v_1,\dots,v_m$ and prescribed cone angles $\Theta_i\in(0,2\pi)$, one seeks a triangulation $\mathcal{T}$ of $\Sigma$ whose vertex set contains all $v_i$. Assigning to each vertex $v$ a target curvature $\overline{K}_v = 2\pi - \Theta_i$ if $v=v_i$ and $\overline{K}_v = 0$ otherwise, the problem reduces to finding a circle packing metric $r\in\mathbb{R}_{>0}^V$ such that $\mathcal{K}(r) = \overline{K}$. If such a circle packing metric exists, the resulting cone metric on $\Sigma$ has the desired cone angles. Thus the existence problem for orbifold metrics translates into characterizing the image of the curvature map $\mathcal{K}$ on a suitable triangulation---the central task addressed by the Circle Pattern Theorem and its generalizations.

The necessity of relaxing the simplicial condition also emerges naturally in computational geometry. In intrinsic geometry processing, powerful algorithms such as intrinsic Delaunay refinement \cite{BS,SSC} routinely generate intermediate triangulations containing loops and multiple edges. Although such singularities are rare on typical inputs, supporting them is essential for the robustness and correctness of these algorithms. Combinatorially, these triangulations are Delta complexes, which form a superset of simplicial complexes; consequently, any algorithm that operates on general Delta complexes automatically applies to ordinary triangle meshes as well. Halfedge mesh data structures, such as those employed in the GeometryCentral library \cite{geometrycentral}, explicitly support Delta complexes precisely to accommodate these intrinsic algorithms with minimal additional implementation effort. Our quasi-simplicial triangulations are precisely the two-dimensional instances of such Delta complexes. This convergence of needs---from pure discrete differential geometry to applied computational geometry---strongly motivates a systematic extension of circle pattern theory to the quasi-simplicial setting.

The purpose of this paper is to relax the simplicial assumption and extend the Circle
Pattern Theorem to quasi-simplicial triangulations.  Our main tool is the finite
covering technique made available by the fact that every quasi-simplicial triangulation
can be lifted to a simplicial one on a suitable cover.

\subsection{Our contribution}
A quasi-simplicial triangulation $\mathcal{T}$ of a closed orientable surface $\Sigma$
is a triangulation whose lift to the universal cover is simplicial; by
\cite[Lemma~2.16]{BW25}, this is equivalent to the existence of a finite covering
$\hat p\colon\widehat{\Sigma}\to\Sigma$ such that $\hat p^*\mathcal{T}$ is a simplicial
triangulation.

As noted in Section~1.1, Chow and Luo~\cite{Chow-Luo} already worked with
quasi-simplicial triangulations, and they observed that, under their restriction
that any three vertices determine at most one triangle, their results on
combinatorial Ricci flow extend to this setting.  However, their treatment remains entirely within the
language of simplicial complexes.  In particular, the KAT inequalities that one
would write down directly on the base quasi-simplicial triangulation $\mathcal{T}$
(without Chow--Luo's restriction) are, in general, not sufficient to characterize
the image of the curvature map; we shall make this point precise in
Section~\ref{sec:why-covering}.

Our approach is fundamentally different.  By placing quasi-simplicial triangulations
in their natural combinatorial and topological framework---that of Delta
complexes---and systematically using the covering theory made available by
Lemma~\ref{lem:quasi-simplicial}, we are able to lift the problem to a genuine
simplicial triangulation on a finite cover, apply the full strength of the classical
Circle Pattern Theorem there, and then push the resulting curvature conditions back
to the base surface.  This yields a complete and explicit characterization of the
curvature image: the correct KAT inequalities must be imposed on \emph{all} subsets
of the \emph{lifted} vertex set~$\widehat{V}$, not merely on those that are preimages of
subsets of~$V$.  These additional constraints constitute the main novel contribution
of this paper.

We are now ready to state the main theorem.

\begin{theorem}\label{T-1-2}
Let $\mathcal{T}$ be a quasi-simplicial triangulation of a closed orientable surface $\Sigma$ and let $\hat p\colon\widehat{\Sigma}\to\Sigma$ be a finite covering such that the lifted triangulation $\hat p^*\mathcal{T}=(\widehat{V},\widehat{E},\widehat{F})$ is simplicial.  
Suppose $\Phi\colon E\to[0,\pi)$ satisfies the condition~\eqref{cond:S}. Then the curvature map $\mathcal{K}$ has the following image: a vector $(K_v)_{v\in V}$ belongs to $\mathcal{K}(\mathbb{R}_{>0}^V)$ if and only if  
\begin{equation}\label{eq:KAT-quasi}
\sum_{\hat v\in \hat I} K_{\hat p(\hat v)} \;>\; -\sum_{(\hat e,\hat\triangle)\in\operatorname{Lk}(\hat I)}(\pi-\Phi(\hat e)) + 2\pi\chi(\widehat\Sigma(\hat I))
\end{equation}
for every non-empty proper subset $\hat I\subset\widehat{V}$, and
\[
\sum_{v\in V}K_v = 2\pi\chi(\Sigma)\text{ in the Euclidean background},\quad
\sum_{v\in V}K_v > 2\pi\chi(\Sigma)\text{ in the hyperbolic background}.
\]
\end{theorem}

Here $\operatorname{Lk}(\hat I)$ and $\widehat\Sigma(\hat I)$ are defined in the same way as before, but with
respect to the lifted simplicial triangulation $\hat p^*\mathcal{T}$. Since every quasi-simplicial triangulation of
the sphere is already simplicial, Theorem~\ref{T-1-2} extends the Circle Pattern
Theorem mainly for closed surfaces of genus $g>0$.

The paper is organized as follows. Section~\ref{S-2} recalls the necessary background on three-circle configurations and gives a precise description of quasi-simplicial triangulations in the language of Delta complexes. Section~\ref{S-3} contains the core of the proof: we introduce the pullback–pushforward formalism, establish the relation between curvatures on the base and on the cover, prove Theorem~\ref{T-1-2}, and finally explain (Section~\ref{sec:why-covering}) why a direct approach without passing to a cover fails to capture the full set of constraints.

\section{Preliminaries on circle configurations and triangulations}\label{S-2}

\subsection{Three-circle configurations}

The geometric realization of a single triangle in a circle pattern depends on the existence and properties of a configuration of three mutually intersecting circles. In this subsection, we collect the essential lemmas concerning such three-circle configurations in both Euclidean and hyperbolic backgrounds.

\begin{lemma}\label{L-2-1}
There exists a configuration of three mutually intersecting circles in both Euclidean and hyperbolic background, unique up to isometry, having radii $r_i,r_j,r_k$  and meeting in exterior intersection angles $\Phi_i, \Phi_j, \Phi_k \in [0,\pi)$ for any three positive numbers $r_i,r_j,r_k$ if and only if $\Phi_i, \Phi_j, \Phi_k$ satisfy
\begin{equation}\label{cond:W}
\Phi_i+\Phi_j+\Phi_k\ \leq \ \pi
\;\;\;\;
\text{or}
\;\;\;\;
\Phi_i+\Phi_j <\pi+\Phi_k,\;\;\Phi_j+\Phi_k < \pi+\Phi_i,\;\;\Phi_k+\Phi_i < \pi+\Phi_j.
\end{equation}
\end{lemma}
\begin{proof}
For the hyperbolic background, this was proved in~\cite{Zhou17}, while for the Euclidean background, the sufficiency was established in~\cite{ZZ24}. We only establish the necessity part for the Euclidean case. Consider the quantity
\begin{align*}
  E(r_i,r_j,r_k)=&4r_i^2r_j^2(1-c_k^2)+4r_j^2r_k^2(1-c_i^2)+4r_k^2r_i^2(1-c_j^2)\\
  &+8r_ir_jr_k\bigl(r_i(c_i+c_jc_k)+r_j(c_j+c_kc_i)+r_k(c_k+c_ic_j)\bigr),
\end{align*}
whose positivity is equivalent to the
triangle inequality for the three lengths $l_{ij},l_{jk},l_{ki}$. Here, $c_s = \cos\Phi_s$ for $s\in\{i,j,k\}$.

Normalizing by $r_k=1$ and $r_i=r_j=t>0$, the condition $E\le 0$ reduces to a quadratic inequality $f(t)=A t^2 + B t + C \ge 0$, where
\[
A = c_k^2 - 1,\quad B = -2(1+c_k)(c_i+c_j),\quad C = (c_i-c_j)^2 - 2(1+c_k).
\]
Since $c_k \neq 1$ and $|c_k|<1$ in the setting of interest, one has $A<0$. A direct computation gives $\Delta = B^2 - 4AC = 8(1+c_k)\bigl((c_ic_j+c_k)^2 - (1-c_i^2)(1-c_j^2)\bigr) \ge 0$. Hence the parabola $f(t)$ opens downward and has real roots, attaining its maximum at
\[
t_{\mathrm{v}} = -\frac{B}{2A} = -\frac{(1+c_k)(c_i+c_j)}{1-c_k^2}.
\]
If condition~\eqref{cond:W} fails, then $\Phi_i+\Phi_j > \pi$, which yields $c_i+c_j < 0$ and thus $t_{\mathrm{v}} > 0$. Consequently, $f(t_{\mathrm{v}}) \ge 0$, so there exists $t>0$ with $f(t)\ge 0$, implying $E(t,t,1)\le 0$ and violation of the triangle inequality. This completes the proof.
\end{proof}

\begin{remark}\label{rem:S-implies-W}
Condition~\eqref{cond:S} strictly implies condition~\eqref{cond:W}. This follows from spherical trigonometry: the angles $\Phi_i,\Phi_j,\Phi_k$ can be interpreted as side lengths of a spherical triangle, and the inequality~\eqref{cond:S} is equivalent to the spherical triangle inequality. A detailed proof can be found in~\cite[Lemma~2.5]{Zhou17}. Consequently, under the assumption~\eqref{cond:S}, for \emph{any} triple of positive radii $(r_i, r_j, r_k)$ the three circles form a non-degenerate geometric triangle.
\end{remark}

\begin{figure}[htbp]
  \centering
  \input{3f}  
  \caption{A three-circle configuration with intersection angles $\Phi_i,\Phi_j,\Phi_k$ and corresponding inner angles $\theta_i,\theta_j,\theta_k$ of the triangle of centers.}
\end{figure}
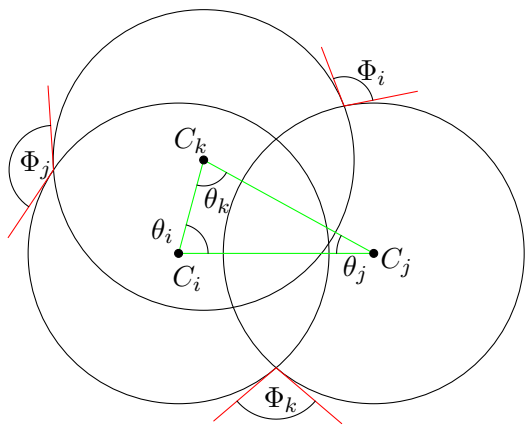

As shown in Figure~1, the geometric centers of the three mutually intersecting circles form a triangle. Let $\theta_i,\theta_j,\theta_k$ denote the corresponding inner angles of this triangle of centers, as indicated. The following limit relations hold in both the Euclidean and hyperbolic backgrounds.

\begin{lemma}[{\cite[Lemma~2.8]{Zhou17}}]\label{L-2-3}
Let $\Phi_i,\Phi_j,\Phi_k$ be as above. Then
\begin{equation*}
\displaystyle{\lim_{r_i\to+\infty}\theta_i\,=\,0,}
\end{equation*}
\begin{equation*}
\displaystyle{\lim_{(r_i,r_j,r_k)\to(0,a,b)}\theta_i\,=\,\pi-\Phi_i,}
\end{equation*}
\begin{equation*}
\displaystyle{\lim_{(r_i,r_j,r_k)\to(0,0,c)}\theta_i+\theta_j\,=\,\pi,}
\end{equation*}
\begin{equation*}
\displaystyle{\lim_{(r_i,r_j,r_k)\to(0,0,0)}\theta_i+\theta_j+\theta_k\,=\,\pi,}
\end{equation*}
where $a,b,c$ are positive constants.
\end{lemma}

The monotonicity of the inner angles with respect to the radii is a cornerstone of the variational approach to circle patterns. In both the Euclidean and hyperbolic backgrounds, one has the following characterization.

\begin{lemma}[{\cite[Lemma~2.6]{Xu} and \cite[Lemma~4.1]{Zhou17}}]\label{L-2-4}
For any three positive radii $(r_i,r_j,r_k)$, the inner angles satisfy the monotonicity properties
\[
\frac{\partial\theta_i}{\partial r_i} < 0,\quad 
\frac{\partial\theta_j}{\partial r_i} \ge 0,\quad 
\frac{\partial(\theta_i+\theta_j+\theta_k)}{\partial r_i} < 0,
\]
if and only if the three intersection angles $\Phi_i,\Phi_j,\Phi_k$ satisfy
\[
\cos\Phi_i + \cos\Phi_j\cos\Phi_k \ge 0,
\]
and its cyclic permutations. These inequalities are precisely the restriction of condition~\eqref{cond:S} to a single triangle.
\end{lemma}

\subsection{Quasi-simplicial triangulations}
Let $\Sigma$ be a closed orientable surface. The following definition of a general triangulation allowing loops and multiple edges is taken from~\cite{BW25}.
\begin{definition}
A \emph{triangulation} $\sT$ of a surface $\Sigma$ consists of a triple $(V,E,F)$ such that:
\begin{itemize}
\item $V$ is a discrete subset of $\Sigma$, whose elements are called vertices;
\item $E$ is a collection of embedded arcs or loops whose endpoints lie in $V$, called edges;
\item $F$ is the set of connected components of $\Sigma \setminus \bigcup_{e\in E} e$, called faces.
\end{itemize}
We further require that:
\begin{itemize}
\item distinct edges intersect only at their endpoints;
\item every vertex is incident to at least three edges (its \emph{degree} $\deg(v)\ge 3$);
\item each face is a topological disk whose boundary consists of exactly three edges.
\end{itemize}
\end{definition}

We regard triangulations as combinatorial data, hence defined up to isotopy of $\Sigma$. Edges are allowed to be loops or parallel (i.e., sharing the same endpoints), so the $1$-skeleton $\Sigma^{(1)}=(V,E)$ is generally a pseudograph. In particular, such a triangulation does not necessarily induce a simplicial complex structure on $\Sigma$.
\begin{definition}
A \emph{quasi-simplicial} triangulation of $\Sigma$ is a triangulation $\mathcal{T}$ such that its lift to the universal covering space $\widetilde{\Sigma}$ is a simplicial triangulation (i.e., contains no loops and no multiple edges).
\end{definition}

A fundamental property of quasi-simplicial triangulations is that they can always be ``unwrapped'' to a simplicial triangulation by passing to a suitable finite covering. The following lemma, due to Bonsante and Wolf, makes this precise.

\begin{lemma}[{\cite[Lemma~2.16]{BW25}}]\label{lem:quasi-simplicial}
For a quasi-simplicial triangulation $\mathcal{T}$ of $\Sigma$, there exists a finite covering $\hp \colon \widehat{\S} \longrightarrow \S$ such that $\hp^*\mathcal{T}$ is a simplicial triangulation.
\end{lemma}

If the covering is chosen to be unramified (which can always be arranged), then $\widehat{\Sigma}$ is a closed orientable surface of genus $\hat g = \deg(\hat p)(g-1)+1$.

\subsection{Delta complex structure}
A quasi-simplicial triangulation can be viewed as a Delta complex, which we now describe. (We refer the reader to Hatcher~\cite[Section~2.1]{Hatcher} for the general theory of Delta complexes and simplicial homology, and to Friedman~\cite{Friedman12} for an elementary illustrated introduction to simplicial sets.)

Denote the triangulation by \(\mathcal{T}=(V,E,F)\), where \(V\) is the set of vertices, \(E\) the set of edges, and \(F\) the set of triangles (faces). We equip the edge set \(E\) with two face maps \(d_0, d_1: E \to V\) which assign to each edge its terminal and initial vertices, respectively (they may coincide, allowing loops). Similarly, the triangle set \(F\) is equipped with three face maps \(d_0, d_1, d_2: F \to E\) sending each triangle to the three edges opposite its zeroth, first, and second vertices (these edges are not required to be distinct). This reflects the intuitive idea that \(d_i\) (in any dimension where it makes sense) is the face map that omits the \(i\)-th vertex of a simplex. If one wished to be perfectly precise, one would write \(d^1_0,d^1_1\) for the face maps from \(E\) to \(V\), and \(d^2_0,d^2_1,d^2_2\) for those from \(F\) to \(E\). However, it is customary to omit the dimension superscript and write simply \(d_i\), since the context makes clear which level is intended (the map is defined whenever the dimension is at least \(i\)). We shall adopt this convention.

For a general Delta complex, these face maps are required only to satisfy the \emph{simplicial identities}: for any \(0\le i<j\le 2\) we have
\[
d_i\circ d_j = d_{j-1}\circ d_i
\]
(whenever the compositions make sense). That is, the diagram below commutes.
\[
\begin{tikzcd}[column sep=large, row sep=large]
F \ar[r, "d_j"] \ar[d, "d_i"'] & E \ar[d, "d_i"] \\
E \ar[r, "d_{j-1}"'] & V
\end{tikzcd}
\]

Let \(v_i(\triangle)\) denotes the \(i\)-th vertex of triangle \(\triangle\) (ordered according to the face maps \(d_i\); it can be obtained by composing the face maps: \(v_0=d_1\circ d_2,v_1=d_0\circ d_2,v_2=d_0\circ d_1\), but we simply assume such a well-defined ordered triple \(\big(v_0(\triangle),v_1(\triangle),v_2(\triangle)\big)\) is given).

We introduce the following notation for incidence relations:
\begin{itemize}
    \item For a vertex \(v\in V\) and an edge \(e\in E\), we write \(v\prec e\) if there exists a face map \(d_i:E\to V\) such that \(d_i(e)=v\) (i.e. \(v\) is an endpoint of \(e\)).
    \item For an edge \(e\in E\) and a triangle \(\triangle\in F\), we write \(e\prec\triangle\) if there exists a face map \(d_i:F\to E\) such that \(d_i(\triangle)=e\) (i.e. \(e\) is a side of \(\triangle\)).
    \item For a vertex \(v\in V\) and a triangle \(\triangle\in F\), we write \(v\prec\triangle\) if there exists an edge \(e\in E\) such that \(v\prec e\) and \(e\prec\triangle\). Equivalently, \(v\prec\triangle\) if and only if \(v=v_i(\triangle)\) for some \(i\in\{0,1,2\}\), i.e. \(v\) is a vertex of \(\triangle\).
\end{itemize}

In the case we are interested in, namely quasi-simplicial triangulations of closed surfaces, one further imposes that each edge is incident to exactly two triangles, i.e. $\#\{ \triangle \in F \mid e \prec \triangle \} = 2$ for every $e\in E$. We can define the geometric realization \(|\mathcal{T}|\) of \(\mathcal{T}\), which is homeomorphic to \(\Sigma\). Concretely, it is obtained as the quotient space
\[
|\mathcal{T}| = \left( V \times \Delta^0 \;\bigsqcup\; E\times \Delta^1 \;\bigsqcup\; F\times \Delta^2 \right) / \sim,
\]
where \(\sim\) is the equivalence relation defined by the following identifications:
\[
\begin{aligned}
&\text{For each } e\in E,\ i\in\{0,1\}: \quad (e, d^i(1)) \sim (d_i(e), 1);\\[4pt]
&\text{For each } \triangle\in F,\ i\in\{0,1,2\}: \quad (\triangle, d^i(s,1-s)) \sim (d_i(\triangle), (s,1-s)).
\end{aligned}
\]
Here \(\Delta^n\) denotes the standard \(n\)-simplex, and \(d^i: \Delta^{n-1} \hookrightarrow \Delta^n\) is the inclusion of the \(i\)-th face, defined by
\[
d^i(t_0,\dots,t_{n-1}) = (t_0,\dots,t_{i-1},0,t_{i+1},\dots,t_{n-1}).
\]

To the Delta complex \(\mathcal{T}\) we associate a chain complex \(C_*(\mathcal{T})\) with
\[
C_0=\mathbb{Z}\langle V\rangle ,\quad C_1=\mathbb{Z}\langle E\rangle ,\quad C_2=\mathbb{Z}\langle F\rangle ,
\]
and boundary maps \(\partial_n:C_n\to C_{n-1}\) defined on generators by
\[
\partial_1(e)=d_1(e)-d_0(e),\qquad 
\partial_2(\triangle)=d_0(\triangle)-d_1(\triangle)+d_2(\triangle),
\]
extended linearly. 

The simplicial identities \(d_i\circ d_j = d_{j-1}\circ d_i\) (for \(0\le i<j\le 2\)) give \(\partial_1\circ\partial_2 = 0\), yielding a chain complex
\[
0 \longrightarrow C_2 \xrightarrow{\partial_2} C_1 \xrightarrow{\partial_1} C_0 \longrightarrow 0.
\]

For a quasi-simplicial triangulation of a closed surface, as discussed above, each edge is incident to exactly two triangles. Hence for every \(e\in E\) there exist unique triangles \(\triangle,\triangle'\in F\) and indices \(i,j\in\{0,1,2\}\) such that \(d_i(\triangle)=e\) and \(d_j(\triangle')=e\). The orientability of \(\Sigma\) is encoded combinatorially by the existence of a function \(\varepsilon:F\to\{\pm\}\) satisfying for every such edge:
\[
\varepsilon(\triangle)\,(-1)^i + \varepsilon(\triangle')\,(-1)^j = 0.
\]
This condition is exactly equivalent to \(\partial_2\big(\sum_{\triangle\in F}\varepsilon(\triangle)\,\triangle\big)=0\). Hence the 2-chain \(\sum_{\triangle\in F}\varepsilon(\triangle)\,\triangle\) is a cycle, and it generates the second homology group \(H_2(\Sigma)\cong\mathbb{Z}\), representing the fundamental class of the surface.

\section{Passing to finite coverings}\label{S-3}

In this section we explain how the study of the curvature map on a quasi-simplicial
triangulation can be reduced to the corresponding problem on a suitable finite
covering triangulation. We denote by $\mathcal{K}\colon\mathbb{R}_{>0}^V\to\mathbb{R}^V$ the curvature map on the original triangulation $\mathcal{T}$, and by $\widehat{\mathcal{K}}$ the analogous curvature map on the lifted triangulation $\widehat{\mathcal{T}}$ (see Subsection~\ref{subsec:metric-to-curvature} for the precise definitions). The key point is that the basic curvature formula is
natural with respect to pullback/pushforward along the covering, so that
\emph{the range of} $\mathcal{K}$ on the original triangulation can be read
off from the range of $\widehat{\mathcal{K}}$ restricted to the pulled-back
radii.

We fix an intersection angle function
\[
\Phi : E \to [0,\pi),
\]
assigning to each edge $e\in E$ a non-negative angle $\Phi(e)$ which
specifies the intersection angle of the two circles centered at the
endpoints of $e$. In this section we work under the condition~\eqref{cond:S}.

\subsection{From circle packing metrics to discrete curvature} \label{subsec:metric-to-curvature}
Let \(\mathcal{T}=(V,E,F)\) be a quasi-simplicial triangulation of a closed orientable surface \(\Sigma\). For each edge \(e\in E\) we are given a fixed intersection angle \(\Phi(e)\in[0,\pi)\). A \emph{circle packing metric} on \(\mathcal{T}\) is a function
\[
r:V\longrightarrow\mathbb{R}_{>0},\qquad r=(r_v)_{v\in V}\in\mathbb{R}_{>0}^V,
\]
assigning a positive radius \(r_v\) to each vertex. For every edge \(e\in E\) we define its length \(l_e\) by the cosine law in the chosen background geometry:
\[
l_e=\begin{cases}
\sqrt{r_{d_0(e)}^2+r_{d_1(e)}^2+2r_{d_0(e)}r_{d_1(e)}\cos\Phi(e)}, & \text{Euclidean background},\\[4pt]
\operatorname{arccosh}\bigl(\cosh r_{d_0(e)}\cosh r_{d_1(e)}+\sinh r_{d_0(e)}\sinh r_{d_1(e)}\cos\Phi(e)\bigr), & \text{hyperbolic background}.
\end{cases}
\]
Thus each edge receives a geometric length depending smoothly on the radii. Now consider a triangle \(\triangle\in F\). Its three edges are \(d_0(\triangle),d_1(\triangle),d_2(\triangle)\in E\), so it determines a geometric triangle with side lengths \(l_{d_0(\triangle)},l_{d_1(\triangle)},l_{d_2(\triangle)}\). Since each geometric triangle is homeomorphic to standard 2-simplex $\Delta^2$ and each geometric edge is homeomorphic to standard 1-simplex $\Delta^1$ via centroid coordinates, we can perform identifications analogous to those in the geometric realization $|\sT|$. This yields a cone metric on the surface \S.

Inside a geometric triangle $\triangle$ we denote the interior angle at the vertex opposite edge \(d_i(\triangle)\) by \(\theta_i^\triangle(r)\). Concretely, if we set
\[
l_i=l_{d_i(\triangle)},\quad l_j=l_{d_j(\triangle)},\quad l_k=l_{d_k(\triangle)}\quad\text{with }\{i,j,k\}=\{0,1,2\},
\]
then the interior angle $\theta_i^\triangle(r)$ is given by the law of cosines:
\[
\theta_i^\triangle(r)=\begin{cases}
\arccos\!\left(\dfrac{l_j^2+l_k^2-l_i^2}{2l_jl_k}\right), & \text{Euclidean background},\\[6pt]
\arccos\!\left(\dfrac{\cosh l_j\cosh l_k-\cosh l_i}{\sinh l_j\sinh l_k}\right), & \text{hyperbolic background}.
\end{cases}
\]

For a vertex \(v\in V\) we need to collect all angles at \(v\) coming from the incident triangles. Since a triangle may contain the same vertex several times (as in the one-vertex triangulation of the torus), we must specify not only the triangle but also the position of the vertex within that triangle. Recall that \(v_i(\triangle)\) denotes the \(i\)-th vertex of triangle \(\triangle\). Then the discrete curvature at \(v\) is defined by
\begin{equation}
{K}_v=2\pi-\sum_{v_i(\triangle)=v}\theta_i^\triangle  ,
\end{equation}
where the sum is taken over all pairs $(\triangle,i)$ with $\triangle\in F$ and $i\in\{0,1,2\}$ such that the $i$-th vertex of $\triangle$ is $v$.
When \sT is simplicial, each triangle contains any given vertex at most once, so we may unambiguously write $\theta_v^\triangle$ for the interior angle at the vertex $v$ in the triangle $\triangle$. Then the curvature formula simplifies to
\begin{equation}
{K}_v=2\pi-\sum_{\triangle\succ v}\theta_v^\triangle .
\end{equation}

Thus we obtain the curvature map $\mathcal{K}\colon \mathbb{R}_{>0}^V \to \mathbb{R}^V$, $r \mapsto (K_v)_{v\in V}$.  The same construction applied to the lifted triangulation $\widehat{\mathcal{T}}$ yields the curvature map $\widehat{\mathcal{K}}\colon \mathbb{R}_{>0}^{\widehat{V}} \to \mathbb{R}^{\widehat{V}}$. 

The curvature map $\mathcal{K}$ is smooth. A circle packing metric $r$ is called \emph{flat} if \(\mathcal{K}(r)=0\). In this case the cone metric has no conical singularities, hence it is a complete Riemannian metric of constant curvature (Euclidean or hyperbolic, respectively) on \(\Sigma\).

\subsection{Pullback of radii and push-forward of curvature}
Let $\hat p \colon \widehat{\Sigma} \to \Sigma$ be a finite-sheeted covering, as given by Lemma~\ref{lem:quasi-simplicial}, whose restriction to the triangulation is a map $\hat p \colon \widehat{\mathcal{T}} \to \mathcal{T}$ with $\widehat{\mathcal{T}} = (\widehat{V},\widehat{E},\widehat{F})$ a simplicial triangulation.
The covering $\hat p$ sends vertices to vertices, edges to edges, and triangles to triangles, and restricts to a bijection on each triangle $\triangle\in \widehat{F}$ onto its image $\hat p(\triangle)\in F$. 

Moreover, $\hat p$ is compatible with the face maps: for every admissible index $i$, the relation $d_i \circ \hat p = \hat p \circ d_i$ holds wherever both sides are defined.  In other words, $\hat p$ preserves the entire incidence structure — it maps the $i$-th vertex of a lifted triangle to the $i$-th vertex of its image, the $i$-th edge to the $i$-th edge, and similarly for the boundary maps.  Consequently, $\hat p$ is a morphism of Delta complexes, respecting the combinatorial identifications that glue the simplices together.

We pull back the intersection angles to $\widehat{\mathcal{T}}$ by declaring that for each edge $\hat e\in\widehat{E}$,
\[
\Phi(\hat e) \coloneqq \Phi\bigl(\hat p(e)\bigr),
\]
so that $\Phi\colon\widehat{E}\to[0,\pi)$ also satisfies condition~\eqref{cond:S}. In the same spirit, for any circle packing metric $r\in\mathbb{R}_{>0}^V$ on $\mathcal{T}$ we define its pullback $\hat p^*r\in\mathbb{R}_{>0}^{\widehat{V}}$ by
\[
(\hat p^*r)_{\hat v} \coloneqq r_{\hat p(\hat v)},\qquad \hat v\in\widehat{V}.
\]

We henceforth consider only those radius assignments on $\widehat{\mathcal{T}}$ that are pullbacks from $\mathcal{T}$; that is, we restrict the curvature map to the subspace
\[
\hat p^\ast(\mathbb{R}_{>0}^V) \;\subset\; \mathbb{R}_{>0}^{\widehat{V}}.
\]

For each triangle $\hat{\triangle} \in \widehat{F}$ with image $ \hat p(\hat{\triangle})\in F$, the geometric construction on $\hat{\triangle}$ with radii $\hat p^\ast r$ and intersection angles ${\Phi}$ is identical to the geometric construction on $\hat p(\hat{\triangle})$ with radii $r$ and angles $\Phi$. Consequently, the interior angles satisfy
\[
{\theta}_{\hat{v}}^{\hat{\triangle}}(\hat p^\ast r)
  \;=\;
  \theta_{i}^{\hat p(\hat\triangle)}(r),\quad
  \hat{v}\prec \hat{\triangle},
\]
where $v_i(\hat\triangle)=\hat v$.

In particular, for a pulled-back circle packing metric $\hat p^\ast r$ for ${r}\in \mathbb{R}_{>0}^{{V}}$,
we have
\begin{equation}\label{eq:Khat-on-pullback}
  \widehat{\mathcal{K}}(\hat p^\ast r)_{\hat{v}}
  \;=\;
  2\pi \;-\; \sum_{\hat\triangle\succ \hat v}\theta_{\hat v}^{\hat\triangle}(\hat p^\ast r)
  \;=\;
  2\pi \;-\; \sum_{v_i(\hat\triangle)=\hat v}\theta_{i}^{\hat p(\hat\triangle)}(r),
  \quad \hat{v}\in \widehat{V}.
\end{equation}

Let $\deg(\hat p)$ denote the degree of the finite covering $\hat p$. Since $\hat p$ is finite-to-one on vertices, we define a \emph{normalized} push-forward map on functions
\[
\hat p_\ast \colon \mathbb{R}^{\widehat{V}} \longrightarrow \mathbb{R}^{V}
\]
by averaging along the fibers:
\[
(\hat p_\ast {f})_v
  \;\coloneqq\;
  \frac{1}{\deg(\hat p)}
  \sum_{ \hat p(\hat{v})=v}
  {f}_{\hat{v}},
  \quad v\in V.
\]
In particular, $\hat p_\ast$ sends the constant function $1$ on $\widehat{V}$ to the constant function $1$ on $V$.

We apply this to the curvature vector $\widehat{\mathcal{K}}(\hat p^\ast r)$.
Using \eqref{eq:Khat-on-pullback}, we obtain
\begin{align*}
  (\hat p_\ast\, \widehat{\mathcal{K}}(\hat p^\ast r))_v
  &= \frac{1}{\deg(\hat p)}
     \sum_{\hat p(\hat{v})=v}
     \Bigl(
       2\pi \;-\; \sum_{\hat \triangle\succ \hat v}\theta_{\hat v}^{\hat \triangle}(\hat p^\ast r)
     \Bigr) \\
  &= \frac{1}{\deg(\hat p)}
     \sum_{\hat p(\hat{v})=v}
     \Bigl(
       2\pi - \sum_{v_i(\hat \triangle)=\hat v}\theta_{i}^{\hat p(\hat\triangle)}(r)
     \Bigr) \\
  &= 2\pi
     \;-\;
     \frac{1}{\deg(\hat p)}
     \sum_{\hat p(\hat{v})=v}
     \sum_{v_i(\hat\triangle)=\hat v}\theta_{i}^{\hat p(\hat\triangle)}(r) .
\end{align*}
On the other hand, for each triangle $\triangle\in F$ containing $v$,
its preimages under $\hat p$ form a finite collection of triangles
$\hat{\triangle}\in\widehat{F}$ with $\hat p(\hat{\triangle})=\triangle$,
and each such $\hat{\triangle}$ contains exactly as many lifts
$\hat{v}$ of $v$. Using the uniformity and local homeomorphism property of the covering, one checks that
\begin{equation}\label{cone-angle}
  \frac{1}{\deg(\hat p)}
     \sum_{\hat p(\hat{v})=v}
     \sum_{v_i(\hat\triangle)=\hat v}\theta_{i}^{\hat p(\hat\triangle)}(r)
  \;=\;
  \sum_{v_i(\triangle)=v}\theta_{i}^{\triangle}(r).
\end{equation}
Indeed, for any $\hat v$ with $\hat p(\hat v)=v$, the condition $v_i(\hat\triangle)=\hat v$ is equivalent to $v_i\big(\hat p(\hat\triangle)\big)=v$ because $d_i\circ\hat p = \hat p\circ d_i$. The map $\hat\triangle\mapsto\hat p(\hat\triangle)$ is a bijection between triangles containing $\hat v$ and those containing $v$, since $\hat p$ is a local homeomorphism. Therefore the sums on both sides consist of exactly the same terms, and the equality follows.
Hence for every $v\in V$ we obtain
\[
(\hat p_\ast\, \widehat{\mathcal{K}}(\hat p^\ast r))_v
  \;=\;
  2\pi - \sum_{v_i(\triangle)=v}\theta_{i}^{\triangle}(r)
  \;=\;
  \mathcal{K}(r)_v.
\]
In vector form,
\begin{equation}\label{eq:K-pull-push}
  \mathcal{K}(r) \;=\; \hat p_\ast\,\widehat{\mathcal{K}}(\hat p^\ast r),
  \quad r\in \mathbb{R}_{>0}^V.
\end{equation}
Equivalently, we write this schematically as
\[
\mathcal{K} \;=\; \hat p_\ast\,\widehat{\mathcal{K}}\,\hat p^\ast.
\]
In words: the curvature map on $\mathcal{T}$ is obtained by first pulling back
the radii to the cover $\widehat{\mathcal{T}}$, computing curvature there, and
then pushing the curvature back down by averaging over the fibers of $\hat p$.

\subsection{Proof of the Circle Pattern Theorem~\ref{T-1-2}}

We are interested in the range of the curvature map
\[
\mathcal{K} \colon \mathbb{R}_{>0}^V \to \mathbb{R}^V.
\]
From the pushforward relation \eqref{eq:K-pull-push}, the range of $\mathcal{K}$ is exactly the push-forward of the curvature values on the covering, restricted to the pulled-back radii:
\begin{equation}\label{eq:range-reduction}
  \mathcal{K}(\mathbb{R}_{>0}^V)
  \;=\;
  \hat p_\ast\Bigl(
    \widehat{\mathcal{K}}\bigl(\hat p^\ast(\mathbb{R}_{>0}^V)\bigr)
  \Bigr).
\end{equation}
This allows us to study $\mathcal{K}$ via $\widehat{\mathcal{K}}$ when the cover $\widehat{\mathcal{T}}$ has better combinatorial properties.

Let $G$ be the (finite) group of deck transformations of the covering $\hat p\colon\widehat{\Sigma}\to\Sigma$. Each $ g\in G$ induces a simplicial automorphism of $\widehat{\mathcal{T}}$ satisfying $\hat p\circ g=\hat p$ and $ g\circ d_i=d_i\circ g$ where appropriate. Since $\Phi$ on $\widehat{E}$ is defined by pullback via $\hat p$, we have $\Phi( g(\hat e))=\Phi(\hat e)$ for all $\hat e\in\widehat{E}$; hence $\Phi$ is $G$-invariant.

The deck transformation group $G = \operatorname{Deck}(\widehat{\Sigma}/\Sigma)$ is isomorphic to the quotient $\pi_1(\Sigma)/\pi_1(\widehat{\Sigma})$. This can be seen from the covering tower $\widetilde{\Sigma}\to\widehat{\Sigma}\to\Sigma$ where $\widetilde{\Sigma}$ is the universal cover of $\Sigma$, yielding the exact sequence
\[
1 \longrightarrow \pi_1(\widehat{\Sigma}) \longrightarrow \pi_1(\Sigma) \longrightarrow G \longrightarrow 1.
\]

A circle packing metric $\hat r\in\mathbb{R}_{>0}^{\widehat{V}}$ is called \emph{$G$-invariant} if $\hat r_{ g(\hat v)}=\hat r_{\hat v}$ for every $\hat v\in\widehat{V}$ and every $ g\in G$.

\begin{proposition}\label{prop:G-invariance}
Let $\hat r\in\mathbb{R}_{>0}^{\widehat{V}}$ be a circle packing metric on the lifted triangulation $\widehat{\mathcal{T}}$. Then $\hat r$ is $G$-invariant if and only if its curvature vector $\widehat{\mathcal{K}}(\hat r)$ is $G$-invariant.
\end{proposition}

\begin{proof}
If $\hat r$ is $G$-invariant, then for any $g\in G$ and any $\hat v\in\widehat{V}$, the edge lengths and interior angles computed from $\hat r$ are unchanged under the deck transformation $g$, because the intersection angles $\Phi$ are $G$-invariant by construction. Hence $\widehat{\mathcal{K}}(\hat r)_{g(\hat v)} = \widehat{\mathcal{K}}(\hat r)_{\hat v}$, so $\widehat{\mathcal{K}}(\hat r)$ is $G$-invariant.

Conversely, assume $\widehat{\mathcal{K}}(\hat r)$ is $G$-invariant. The curvature map $\widehat{\mathcal{K}}$ is smooth and, by Lemma~\ref{L-2-4}, its Jacobian is everywhere positive definite in the hyperbolic background, and positive semi-definite with a one-dimensional kernel in the Euclidean background (see Chow--Luo~\cite[Proposition~3.9]{Chow-Luo}). Thus $\widehat{\mathcal{K}}$ is injective in the hyperbolic background, and in the Euclidean background its fibers consist of constant rescalings. Suppose, for contradiction, that $\hat r$ is not $G$-invariant. Then there exist $g\in G$ and $\hat v\in\widehat{V}$ such that $\hat r_{g(\hat v)}\neq \hat r_{\hat v}$. In the hyperbolic background, injectivity directly yields $\widehat{\mathcal{K}}(\hat r)\neq\widehat{\mathcal{K}}(g\cdot \hat r)$. In the Euclidean background, if $g\cdot \hat r$ were a scalar multiple of $\hat r$, the fact that $g$ acts as a permutation on $\widehat{V}$ would force $\hat r_{g(\hat v)} = \hat r_{\hat v}$ for every $\hat v$, contradicting the choice of $g$ and $\hat v$. Hence $g\cdot \hat r$ and $\hat r$ lie in different scaling fibers, so again $\widehat{\mathcal{K}}(\hat r)\neq\widehat{\mathcal{K}}(g\cdot \hat r)$. In either case, equivariance gives $\widehat{\mathcal{K}}(g\cdot \hat r)=g\cdot\widehat{\mathcal{K}}(\hat r)$, while $G$-invariance of $\widehat{\mathcal{K}}(\hat r)$ implies $g\cdot\widehat{\mathcal{K}}(\hat r)=\widehat{\mathcal{K}}(\hat r)$, a contradiction. Therefore $\hat r$ must be $G$-invariant.
\end{proof}

Thus on $\widehat{\mathcal{T}}$ the curvature map $\widehat{\mathcal{K}}$ establishes a bijection between $G$-invariant circle packing metrics and $G$-invariant curvature vectors. In particular, the pulled-back radii $\hat p^*(\mathbb{R}_{>0}^V)$ are exactly the $G$-invariant circle packing metrics, and their curvature images under $\widehat{\mathcal{K}}$ are precisely the $G$-invariant curvature vectors. We now proceed to the proof of Theorem~\ref{T-1-2}.

\begin{proof}[Proof of Theorem~\ref{T-1-2}]
Since $\hat p^*(\mathbb{R}_{>0}^V)$ is precisely the set of $G$-invariant circle packing metrics on $\widehat{\mathcal{T}}$, and $\widehat{\mathcal{K}}$ restricts to a bijection between $G$-invariant radii and $G$-invariant curvatures, it follows that $\widehat{\mathcal{K}}(\hat p^*(\mathbb{R}_{>0}^V))$ is exactly the set of $G$-invariant curvature vectors that lie in the image of $\widehat{\mathcal{K}}$.

By Theorem~\ref{T-1-1} applied to the simplicial triangulation $\widehat{\mathcal{T}}$, a vector $\widehat{K}\in\mathbb{R}^{\widehat{V}}$ belongs to $\widehat{\mathcal{K}}(\mathbb{R}_{>0}^{\widehat{V}})$ if and only if it satisfies the KAT inequalities on $\widehat{\mathcal{T}}$ together with the appropriate global Gauss--Bonnet condition. Restricting to $G$-invariant vectors, we conclude that a vector $K\in\mathbb{R}^V$ lies in the image of $\mathcal{K}$ if and only if the following two conditions hold:
\begin{enumerate}
    \item For every non-empty proper subset $\hat I\subsetneq\widehat{V}$, the pulled-back curvature $\hat p^*K$ satisfies
    \[
\sum_{\hat v\in \hat I} (\hat p^*K)_{\hat v} \;>\; -\sum_{(\hat e,\hat \triangle)\in\operatorname{Lk}(\hat I)}(\pi-\Phi(e)) + 2\pi\chi(\widehat \Sigma(\hat I)).
    \]
    \item The total curvature obeys the global Gauss--Bonnet constraint
    \[
    \sum_{\hat v\in \widehat V}(\hat p^*K)_{\hat v} = 2\pi\chi(\widehat\Sigma)\text{ in the Euclidean background},\quad
    \sum_{\hat v\in\widehat  V}(\hat p^*K)_{\hat v} > 2\pi\chi(\widehat\Sigma)\text{ in the hyperbolic background}.
    \]
\end{enumerate}
Since $\chi(\widehat{\Sigma}) = \deg(\hat p)\,\chi(\Sigma)$ and the pushforward $\hat p_\ast$ averages the fibre sum, these conditions translate exactly to those stated in Theorem~\ref{T-1-2}.
\end{proof}

\subsection{Why a covering is necessary}\label{sec:why-covering}

One might attempt to write down KAT inequalities directly on the original
quasi-simplicial triangulation $\mathcal{T}$, using only the combinatorics of
$V,E,F$.  Since $\mathcal{T}$ is now a Delta complex, the link of a vertex subset
should be defined with care.  For a non-empty proper subset $I\subset V$, we set
\[
\operatorname{Lk}(I)=\bigl\{(e,\triangle)\mid \exists i\text{ with }v_i(\triangle)\in I,\;
d_i(\triangle)=e,\; v_s(\triangle)\notin I\text{ for }s\neq i\bigr\}.
\]
By Lemma~\ref{L-2-3}, when all radii $r_v$ for $v\in I$ tend to $0$ while the others
remain fixed, the interior angles converge to the limits prescribed there. Consequently,
one can compute the limit of the total curvature $\sum_{v\in I}K_v$ as
\[
\lim_{r_v\to0\;(v\in I)}\sum_{v\in I}K_v = 2\pi|I| - (|A_2|+|A_3|)\pi
  \;-\; \sum_{(e,\triangle)\in\operatorname{Lk}(I)}(\pi-\Phi(e)),
\]
where $A_2$ (resp.\ $A_3$) denotes the set of triangles $\triangle$ for which exactly
two (resp.\ three) of the indexed vertices $v_0(\triangle),v_1(\triangle),v_2(\triangle)$
belong to $I$; here vertices are counted with multiplicity according to their indices,
so even if some $v_i(\triangle)$ coincide as points of $\Sigma$, each occurrence
contributes separately.  This limit computation is the combinatorial analogue of the
degenerate circle packing limits studied by Thurston~\cite{Thurston}, with the details filled in by
Marden--Rodin~\cite{MR90}, and later refined in the framework of combinatorial
Ricci flow by Chow--Luo~\cite{Chow-Luo}.  By an argument completely analogous
to the one in~\cite{Chow-Luo}, we obtain the KAT inequality
\[
\sum_{v\in I}K_v \;>\; -\sum_{(e,\triangle)\in\operatorname{Lk}(I)}(\pi-\Phi(e))
  \;+\; 2\pi\chi(\Sigma(I)).
\]

However, the collection of such inequalities (one for each
$\varnothing\neq I\subsetneq V$) is \emph{not sufficient} to characterize the image
of $\mathcal{K}$.  Indeed, consider a subset of the lifted vertex set of the form
$\hat I=\hat p^{-1}(I)$ for some $I\subset V$.  For such a subset, the KAT inequality
given by Theorem~\ref{T-1-2} reads
\[
\sum_{\hat v\in\hat p^{-1}(I)} K_{\hat p(\hat v)}
\;>\; -\sum_{(\hat e,\hat\triangle)\in\operatorname{Lk}(\hat p^{-1}(I))}
   (\pi-\Phi(\hat e)) \;+\; 2\pi\chi\bigl(\widehat\Sigma(\hat p^{-1}(I))\bigr).
\]
Because $\hat p$ is a regular covering, one has
$\sum_{\hat v\in\hat p^{-1}(I)} K_{\hat p(\hat v)} = \deg(\hat p)\sum_{v\in I}K_v$,
and basic covering theory gives
\[
\chi\bigl(\widehat\Sigma(\hat p^{-1}(I))\bigr) = \deg(\hat p)\,\chi(\Sigma(I))
\]
and
\[
\sum_{(\hat e,\hat\triangle)\in\operatorname{Lk}(\hat p^{-1}(I))}
   (\pi-\Phi(\hat e)) = \deg(\hat p)
   \sum_{(e,\triangle)\in\operatorname{Lk}(I)}(\pi-\Phi(e)).
\]
Dividing the inequality on the cover by $\deg(\hat p)$ therefore reproduces exactly the
KAT inequality on the base triangulation for the subset $I$.  Thus the inequalities
on $\hat p^*\mathcal{T}$ that correspond to subsets $\hat I$ that are preimages of subsets of~$V$
are precisely equivalent to those one would write down directly on $\mathcal{T}$.

The crucial point is that the lifted triangulation $\hat p^*\mathcal{T}$ possesses many subsets
$\hat I\subset\widehat{V}$ that are \emph{not} of the form $\hat p^{-1}(I)$ for any
$I\subset V$.  These give additional constraints that are invisible when working
purely on $\mathcal{T}$.  Therefore the inequalities obtained directly from
$\mathcal{T}$ are only necessary conditions, and they are in general not sufficient.
This is why the covering approach is essential: it supplies the full set of
constraints that characterize the curvature range.

\section{Acknowledgements}

\vspace{12pt}
The first author is partially supported by National Natural Science Foundation of China (Grant No. 12171480), Hunan Provincial Natural Science Foundation of China (Grant No. 2022JJ10059) and Scientific Research Program of NUDT (Grant No. JS2023-01 and No. IISF-C24001). 

\section*{Data Availability Statement}
Data sharing is not applicable to this theoretical study as no new data were created.


\bibliographystyle{acm}
\bibliography{refs}

\end{document}

%% file: 1vertex-T2.tex
\begin{tikzpicture}[line cap=round,line join=round,>=triangle 45,x=1.0cm,y=1.0cm]
\clip(-0.5,-3.5) rectangle (6.5,1.);
\fill[line width=0.pt,fill=black,fill opacity=0.10000000149011612] (2.,0.) -- (0.,0.) -- (0.,-2.) -- (2.,-2.) -- cycle;
\fill[line width=0.pt,dash pattern=on 1pt off 1pt,fill=black,fill opacity=0.10000000149011612] (3.,0.5) -- (3.,-2.5) -- (6.,-2.5) -- (6.,0.5) -- cycle;
\draw [line width=0.4pt,dash pattern=on 1pt off 1pt] (2.,0.)-- (0.,0.);
\draw [line width=0.4pt] (0.,-2.)-- (2.,0.);
\draw [line width=0.4pt] (0.,-2.)-- (0.,0.);
\draw [line width=0.4pt] (0.,-2.)-- (2.,-2.);
\draw [line width=0.4pt,dash pattern=on 1pt off 1pt] (0.,0.)-- (2.,0.);
\draw [line width=0.4pt,dash pattern=on 1pt off 1pt] (2.,0.)-- (2.,-2.);
\draw [line width=0.4pt,dash pattern=on 1pt off 1pt] (3.,0.5)-- (3.,-2.5);
\draw [line width=0.4pt] (3.,0.5)-- (3.,-0.5);
\draw [line width=0.4pt] (3.,-0.5)-- (3.,-1.5);
\draw [line width=0.4pt] (3.,-1.5)-- (3.,-2.5);
\draw [line width=0.4pt] (3.,-2.5)-- (4.,-2.5);
\draw [line width=0.4pt] (4.,-2.5)-- (5.,-2.5);
\draw [line width=0.4pt] (5.,-2.5)-- (6.,-2.5);
\draw [line width=0.4pt,dash pattern=on 1pt off 1pt] (6.,-2.5)-- (6.,-1.5);
\draw [line width=0.4pt,dash pattern=on 1pt off 1pt] (6.,-1.5)-- (6.,-0.5);
\draw [line width=0.4pt,dash pattern=on 1pt off 1pt] (6.,-0.5)-- (6.,0.5);
\draw [line width=0.4pt,dash pattern=on 1pt off 1pt] (6.,0.5)-- (5.,0.5);
\draw [line width=0.4pt,dash pattern=on 1pt off 1pt] (5.,0.5)-- (4.,0.5);
\draw [line width=0.4pt,dash pattern=on 1pt off 1pt] (4.,0.5)-- (3.,0.5);
\draw [line width=0.4pt] (4.,0.5)-- (4.,-0.5);
\draw [line width=0.4pt] (4.,-0.5)-- (4.,-1.5);
\draw [line width=0.4pt] (4.,-1.5)-- (4.,-2.5);
\draw [line width=0.4pt] (5.,-2.5)-- (5.,-1.5);
\draw [line width=0.4pt] (5.,-1.5)-- (5.,-0.5);
\draw [line width=0.4pt] (5.,-0.5)-- (5.,0.5);
\draw [line width=0.4pt] (5.,-0.5)-- (6.,-0.5);
\draw [line width=0.4pt] (5.,-1.5)-- (6.,-1.5);
\draw [line width=0.4pt] (5.,-1.5)-- (4.,-1.5);
\draw [line width=0.4pt] (4.,-1.5)-- (3.,-1.5);
\draw [line width=0.4pt] (3.,-0.5)-- (4.,-0.5);
\draw [line width=0.4pt] (4.,-0.5)-- (5.,-0.5);
\draw [line width=0.4pt] (3.,-0.5)-- (4.,0.5);
\draw [line width=0.4pt] (5.,0.5)-- (4.,-0.5);
\draw [line width=0.4pt] (4.,-0.5)-- (3.,-1.5);
\draw [line width=0.4pt] (3.,-2.5)-- (4.,-1.5);
\draw [line width=0.4pt] (4.,-1.5)-- (5.,-0.5);
\draw [line width=0.4pt] (5.,-0.5)-- (6.,0.5);
\draw [line width=0.4pt] (6.,-0.5)-- (5.,-1.5);
\draw [line width=0.4pt] (5.,-1.5)-- (4.,-2.5);
\draw [line width=0.4pt] (5.,-2.5)-- (6.,-1.5);
\draw (0.8,-2.6) node[anchor=north west] {$\mathcal{T}$};
\draw (4.3,-2.6) node[anchor=north west] {$\hat{\mathcal{T}}$};
\begin{scriptsize}
\draw [fill=ffffff] (2.,0.) circle (1.5pt);
\draw [fill=ffffff] (0.,0.) circle (1.5pt);
\draw [fill=black] (0.,-2.) circle (1.5pt);
\draw [fill=ffffff] (2.,-2.) circle (1.5pt);
\draw [fill=ffffff] (3.,0.5) circle (1.5pt);
\draw [fill=black] (3.,-2.5) circle (1.5pt);
\draw [fill=ffffff] (6.,-2.5) circle (1.5pt);
\draw [fill=ffffff] (6.,0.5) circle (1.5pt);
\draw [fill=black] (3.,-0.5) circle (1.5pt);
\draw [fill=black] (3.,-1.5) circle (1.5pt);
\draw [fill=black] (4.,-2.5) circle (1.5pt);
\draw [fill=black] (5.,-2.5) circle (1.5pt);
\draw [fill=ffffff] (6.,-1.5) circle (1.5pt);
\draw [fill=ffffff] (6.,-0.5) circle (1.5pt);
\draw [fill=ffffff] (5.,0.5) circle (1.5pt);
\draw [fill=ffffff] (4.,0.5) circle (1.5pt);
\draw [fill=black] (4.,-0.5) circle (1.5pt);
\draw [fill=black] (4.,-1.5) circle (1.5pt);
\draw [fill=black] (5.,-1.5) circle (1.5pt);
\draw [fill=black] (5.,-0.5) circle (1.5pt);
\end{scriptsize}
\end{tikzpicture}

%% file: 3f.tex
\begin{tikzpicture}[line cap=round,line join=round,>=triangle 45,x=1.0cm,y=1.0cm]
\clip(-2.4,-2.3) rectangle (4.65,3.3);
\draw (-0.21765919811320703,-0.03893805309734472) node[anchor=north west] {$C_i$};
\draw (2.55,0.2) node[anchor=north west] {$C_j$};
\draw (-0.2,1.8) node[anchor=north west] {$C_k$};
\draw (0.2,1) node[anchor=north west] {$\theta_k$};
\draw (-0.5,0.6) node[anchor=north west] {$\theta_i$};
\draw (2.05,0.1) node[anchor=north west] {$\theta_j$};
\draw (2.2307193396226426,2.7) node[anchor=north west] {$\Phi_i$};
\draw (-2.25,1.5) node[anchor=north west] {$\Phi_j$};
\draw (1,-1.65) node[anchor=north west] {$\Phi_k$};
\draw [line width=0.4pt] (0.,0.) circle (2.cm);
\draw [line width=0.4pt] (2.59456,0.) circle (2.cm);
\draw [line width=0.4pt] (0.333458,1.24296) circle (2.cm);
\draw [line width=0.4pt,color=qqffqq] (0.333458,1.24296)-- (0.,0.);
\draw [line width=0.4pt,color=qqffqq] (0.,0.)-- (2.59456,0.);
\draw [line width=0.4pt,color=qqffqq] (2.59456,0.)-- (0.333458,1.24296);
\draw [line width=0.4pt,color=ffqqqq] (2.2002183784196476,1.960738301122661)-- (1.8998302237934257,2.7419721505240218);
\draw [line width=0.4pt,color=ffqqqq] (2.2002183784196476,1.960738301122661)-- (3.1758814700510114,2.156962634298444);
\draw [line width=0.4pt,color=ffqqqq] (1.29728,-1.5221907244494692)-- (0.46741219547902135,-2.2294417201079395);
\draw [line width=0.4pt,color=ffqqqq] (1.29728,-1.5221907244494692)-- (2.168823308804934,-2.264959475753864);
\draw [line width=0.4pt,color=ffqqqq] (-1.6622600144255697,1.11215630396177)-- (-1.73541680961735,2.228335268806968);
\draw [line width=0.4pt,color=ffqqqq] (-1.6622600144255697,1.11215630396177)-- (-2.2659443865455895,0.20987270449461293);
\draw [shift={(2.2002183784196476,1.960738301122661)},line width=0.4pt]  plot[domain=0.19847123811168724:1.9378738114356528,variable=\t]({1.*0.39004988360217235*cos(\t r)+0.*0.39004988360217235*sin(\t r)},{0.*0.39004988360217235*cos(\t r)+1.*0.39004988360217235*sin(\t r)});
\draw [shift={(-1.6622600144255697,1.11215630396177)},line width=0.4pt]  plot[domain=1.636244889784653:4.122729364137917,variable=\t]({1.*0.5913903590733618*cos(\t r)+0.*0.5913903590733618*sin(\t r)},{0.*0.5913903590733618*cos(\t r)+1.*0.5913903590733618*sin(\t r)});
\draw [shift={(1.29728,-1.5221907244494692)},line width=0.4pt]  plot[domain=3.8473888283845668:5.577389132384813,variable=\t]({1.*0.6833469777656588*cos(\t r)+0.*0.6833469777656588*sin(\t r)},{0.*0.6833469777656588*cos(\t r)+1.*0.6833469777656588*sin(\t r)});
\draw [shift={(2.59456,0.)},line width=0.4pt]  plot[domain=2.6389688515685665:3.141592653589793,variable=\t]({1.*0.4935360672310792*cos(\t r)+0.*0.4935360672310792*sin(\t r)},{0.*0.4935360672310792*cos(\t r)+1.*0.4935360672310792*sin(\t r)});
\draw [shift={(0.,0.)},line width=0.4pt]  plot[domain=0.:1.3086908001663882,variable=\t]({1.*0.3922795005297576*cos(\t r)+0.*0.3922795005297576*sin(\t r)},{0.*0.3922795005297576*cos(\t r)+1.*0.3922795005297576*sin(\t r)});
\draw [shift={(0.333458,1.24296)},line width=0.4pt]  plot[domain=4.450283453756181:5.780561505158359,variable=\t]({1.*0.345612977927395*cos(\t r)+0.*0.345612977927395*sin(\t r)},{0.*0.345612977927395*cos(\t r)+1.*0.345612977927395*sin(\t r)});
\begin{scriptsize}
\draw [fill=black] (0.,0.) circle (1.5pt);
\draw [fill=black] (2.59456,0.) circle (1.5pt);
\draw [fill=black] (0.333458,1.24296) circle (1.5pt);
\end{scriptsize}
\end{tikzpicture}